\documentclass[12pt,twoside,a4paper]{article}
\usepackage{bm}
\usepackage{braket}
\usepackage{graphicx}
\usepackage{graphics}
\usepackage{theorem}
\usepackage{amsmath}
\usepackage{amssymb,mathrsfs}
\usepackage{latexsym}
\usepackage{cases}
\usepackage[dvipdfmx,bookmarksnumbered,colorlinks,linkcolor=blue,urlcolor=blue,citecolor=blue]{hyperref}
\usepackage{color}
\usepackage{algorithm}
\usepackage{algpseudocode}
\usepackage{caption}
\newfloat{procedure}{H}{loa}
\floatname{procedure}{Procedure} 
\theorembodyfont{\slshape}
\newtheorem{THM}{Theorem}[section]

\newtheorem{PROP}{Proposition}[section]
\newtheorem{COR}{Corollary}[section]

\newtheorem{ASS}{Assumption}[section]

\theorembodyfont{\rmfamily}
\newtheorem{REM}{Remark}[section]

\newtheorem{hm-cond}{Condition}
\usepackage{bbm}   	
\newcommand{\dd}{\mathbbm}

\usepackage{color}

\numberwithin{equation}{section}  
\makeatletter
\def\Left#1#2\Right{\begingroup%
\def\ts@r{\nulldelimiterspace=0pt \mathsurround=0pt}%
\let\@hat=#1%
\def\sht@im{#2}%
\def\@t{{\mathchoice{\def\@fen{\displaystyle}\k@fel}%
{\def\@fen{\textstyle}\k@fel}%
{\def\@fen{\scriptstyle}\k@fel}%
{\def\@fen{\scriptscriptstyle}\k@fel}}}%
\def\g@rin{\ts@r\left\@hat\vphantom{\sht@im}\right.}%
\def\k@fel{\setbox0=\hbox{$\@fen\g@rin$}\hbox{%
$\@fen \kern.3875\wd0 \copy0 \kern-.3875\wd0%
\llap{\copy0}\kern.3875\wd0$}}%
\def\pt@h{\mathopen\@t}\pt@h\sht@im%
\Right}%
\def\Right#1{\let\@hat=#1%
\def\st@m{\mathclose\@t}%
\st@m\endgroup}
\makeatother

\newcommand{\vc}{\bm}
\makeatletter
\DeclareRobustCommand\widecheck[1]{{\mathpalette\@widecheck{#1}}}
\def\@widecheck#1#2{%
\setbox\z@\hbox{\m@th$#1#2$}%
\setbox\tw@\hbox{\m@th$#1%
\widehat{%
\vrule\@width\z@\@height\ht\z@
\vrule\@height\z@\@width\wd\z@}$}%
\dp\tw@-\ht\z@
\@tempdima\ht\z@ \advance\@tempdima2\ht\tw@ \divide\@tempdima\thr@@
\setbox\tw@\hbox{%
\raise\@tempdima\hbox{\scalebox{1}[-1]{\lower\@tempdima\box
\tw@}}}%
{\ooalign{\box\tw@ \cr \box\z@}}}
\makeatother

\newcommand{\ool}[1]{\overline{\overline{\bm{#1}}}}
\newcommand{\ol}{\overline}

\newcommand{\bv}{\breve}

\newcommand{\vertiii}[1]%
{{\left\vert\kern-0.25ex\left\vert\kern-0.25ex\left\vert #1
\right\vert\kern-0.25ex\right\vert\kern-0.25ex\right\vert}}
\newcommand{\down}[2]{\smash{\lower#1\hbox{#2}}}
\newcommand{\up}[2]{\smash{\lower-#1\hbox{#2}}}
\def\PFOF#1{\noindent{\it Proof of {#1}.~~}}
\newcommand{\dm}{\displaystyle}
\newcommand{\qed}{\hspace*{\fill}$\Box$}
\newcommand{\proof}{\noindent {\it Proof:~}}

\newcommand{\vmin}{\wedge}

\newcommand{\EE}{\mathbb{E}}
\newcommand{\PP}{\mathbb{P}}

\newcommand{\bcal}[1]{\bm{\mathcal{#1}}}

\newcommand{\bbA}{\mathbb{A}}

\newcommand{\bbL}{\mathbb{L}}
\newcommand{\bbM}{\mathbb{M}}
\newcommand{\bbN}{\mathbb{N}}

\newcommand{\bbS}{\mathbb{S}}

\newcommand{\bbZ}{\mathbb{Z}}

\newcommand{\varep}{\varepsilon}

\makeatletter
\def\widebar{\accentset{{\cc@style\underline{\mskip10mu}}}}
\def\Widebar{\accentset{{\cc@style\underline{\mskip8mu}}}}
\makeatother
\newcommand{\wbartil}[1]{\if#1L \widebar{\hspace{-0.12zw}\widetilde{#1}}
\else {\if#1M \widebar{\widetilde{\!M\!}}
\else {\if#1W \widebar{\widetilde{\!W\!}}
\else {\if#1U \widebar{\hspace{-0.03zw}\widetilde{#1}}
\else {\if#1V \widebar{\hspace{-0.03zw}\widetilde{#1}}
\else {\if#1Y \widebar{\hspace{-0.0zw}\widetilde{#1}}
\else \,\widebar{\!\widetilde{#1}}
\fi}
\fi}
\fi}
\fi}
\fi}
\fi}
\makeatletter
\def\eqnarray{\stepcounter{equation}\let\@currentlabel=\theequation
\global\@eqnswtrue
\global\@eqcnt\z@\tabskip\@centering\let\\=\@eqncr
$$\halign to \displaywidth\bgroup\@eqnsel\hskip\@centering
$\displaystyle\tabskip\z@{##}$&\global\@eqcnt\@ne
\hfil$\;{##}\;$\hfil
&\global\@eqcnt\tw@ $\displaystyle\tabskip\z@{##}$\hfil
\tabskip\@centering&\llap{##}\tabskip\z@\cr}
\makeatother

\makeatother

\begin{document}\thispagestyle{empty}

\hfill

\vspace{-10mm}

{\large{\bf
\begin{center}
A geometric convergence formula for the level-increment-truncation approximation of M/G/1-type Markov chains%
%
%
%
%
\end{center}
}
}

\begin{center}
{
\begin{tabular}[h]{cc}
Katsuhisa Ouchi\footnotemark[2] & Hiroyuki Masuyama\footnotemark[3]                        \\ 
\textit{Kyoto University}&\textit{Tokyo Metropolitan University}\\ 
\end{tabular}
\footnotetext[2]{E-mail: o-uchi@sys.i.kyoto-u.ac.jp}
\footnotetext[3]{E-mail: masuyama@tmu.ac.jp}
}

\bigskip
\medskip

{\small
\textbf{Abstract}

\medskip

\begin{tabular}{p{0.85\textwidth}}
This paper considers an approximation usually used when implementing Ramaswami's recursion for the stationary distribution of the M/G/1-type Markov chain.  The approximation is called the level-increment-truncation approximation because it truncates level increment at a given threshold. The main contribution of this paper is to present a geometric convergence formula of the level-wise difference between the respective stationary distributions of the original M/G/1-type Markov chain and its LI truncation approximation under the assumption that the level-increment distribution is light-tailed.
\end{tabular}
}
\end{center}

\begin{center}
\begin{tabular}{p{0.90\textwidth}}
{\small

\medskip

}
\end{tabular}

\end{center}

\section{Introduction}
\label{sec:Intro}

This paper considers a computable approximation of M/G/1-type Markov chains from the motivation of controlling the computational error of their stationary distribution vectors. The class of M/G/1-type Markov chains plays an important role in the analysis of M/G/1-type queues \cite{Neut89} including BMAP/GI/1 ones \cite{Luca91}, where the abbreviation ``BMAP" stands for the batch Markovian arrival process.

The (canonical) M/G/1-type Markov chain is driven by the following transition probability matrix (see, e.g., \cite[Canonical form at page 76 in Chapter~2]{Neut89} and \cite[Section~3.5]{He14}):
\begin{equation*}
\vc{P}=
\bordermatrix{
& \bbL_0
& \bbL_1
& \bbL_2
& \bbL_3
& \cdots \cr
\bbL_0
& \vc{B}_0
& \vc{B}_1
& \vc{B}_2
& \vc{B}_3
& \cdots \cr
\bbL_1
& \vc{B}_{-1}
& \vc{A}_0
& \vc{A}_1
& \vc{A}_2
& \cdots \cr
\bbL_2
& \vc{O}
& \vc{A}_{-1}
& \vc{A}_0
& \vc{A}_1
& \cdots \cr
\bbL_3
& \vc{O}
& \vc{O}
& \vc{A}_{-1}
& \vc{A}_0
& \cdots
\cr
~\vdots
& \vdots
& \vdots
& \vdots
& \vdots
& \ddots
},
\end{equation*}
where $\vc{O}$ denotes the zero matrix, $\bbL_0 = \{0\} \times \{1,2,\dots,M_0\}$, and $\bbL_k = \{k\} \times \{1,2,\dots,M_1\}$ for $k \in \bbN:=\{1,2,3,\dots\}$. For $k \in \bbZ_+ :=\{0,1,2,\dots\}$, the subset $\bbL_k$ of the state space $\bbS := \cup_{k=0}^{\infty}\bbL_k$ is called level $k$, and an element $(k, j) \in \bbL_k$ is called {\it phase} $j$ of level $k$.

The M/G/1-type stochastic matrix $\vc{P}$ is irreducible and positive recurrent with the unique stationary distribution vector, denoted by $\vc{\pi} = (\pi_{k, i})_{(k,i) \in \bbS}$, under the well-known assumption below (see \cite[Chapter~XI, Proposition~3.1]{Asmu03}).
\begin{ASS} \label{assum1}
Let $\vc{e}$ denote a column vector of ones. The following hold: (i) The stochastic matrix $\vc{P}$ is irreducible;
(ii) $\vc{A}:=\sum_{k=-1}^\infty \vc{A}_k$ is an irreducible stochastic matrix;
(iii) $\ol{\vc{m}}_B:=\sum_{k = 1}^\infty k\vc{B}_k\vc{e}$ is finite; and
(iv) $\sigma := \vc{\varpi} \ol{\vc{m}}_A < 0$, where $\ol{\vc{m}}_A = \sum_{k = -1}^\infty k\vc{A}_k\vc{e}$, and where $\vc{\varpi}$ denotes the unique stationary distribution vector of $\vc{A}$.
\end{ASS}

\begin{REM}
Throughout the paper, we follow the standard rules of vector notation in Neuts-style matrix analysis \cite{Neut89,Lato99}. That is, row vectors are denoted by bold English small letters and column vectors by bold Greek small letters. The only exception is the stationary distribution vector of the probability matrix $\vc{G}$, which will be defined later. This is a row vector, but it is conventionally assigned the bold English small letter $\vc{g}$.
\end{REM}

It is standard to use Ramaswami's recursion \cite{Rama88} for computing the level-wise partitioned stationary distribution vector $\vc{\pi} = (\vc{\pi}_0,\vc{\pi}_1, \vc{\pi}_2,\dots)$ of the M/G/1-type stochastic matrix $\vc{P}$. The key component of Ramaswami's recursion is the $G$-matrix $\vc{G}$, and $\vc{G}$ is the minimal nonnegative solution of the following matrix equation (see, e.g., \cite{Neut89}):
\begin{align}
\vc{G} = \sum_{m=-1}^{\infty} \vc{A}_m\vc{G}^{m+1}.
\label{eqn:G-matrix}
\end{align}
Under Assumption \ref{assum1}, the matrix $\vc{G}$ is a stochastic matrix \cite[Theorem 2.3.1]{Neut89} and has a single communication class \cite[Proposition~2.1]{Kimu10}. Therefore, $\vc{G}$ has the unique stationary distribution vector, denoted by $\vc{g}$ (the vector $\vc{g}$ is used in the next section). With this matrix $\vc{G}$, Ramaswami's recursion \cite{Rama88} is described in the following way (see also \cite{Sche90}).
\begin{subequations}\label{eq:Rama-pi(k)}
\begin{alignat}{2}
\vc{\pi}_k
&= \vc{\pi}_0\vc{R}_0(k)
+ \sum_{\ell=1}^{k-1}\vc{\pi}_\ell\vc{R}(k-\ell),
& \quad k &\in \bbN,
\\
\vc{R}(k)
&=
\dm\sum_{m=0}^{\infty}
\vc{A}_{k+m}\vc{G}^m[\vc{I}-\vc{\Phi}_0]^{-1}, & \quad k &\in \bbN,
\\
\vc{R}_0(k)
&=
\sum_{m=0}^{\infty}
\vc{B}_{k+m}\vc{G}^m[\vc{I}-\vc{\Phi}_0]^{-1}, & \quad k &\in \bbN,
\\
\vc{\Phi}_0 &= \sum_{m=0}^{\infty}\vc{A}_m\vc{G}^m,
\end{alignat}
\end{subequations}
where $\vc{I}$ denotes the identity matrix. Furthermore, $\vc{\pi}_0$ is given by
\begin{align}
\vc{\pi}_0
= {
\vc{\kappa}
\over
\vc{\kappa}\vc{R}_0 [\vc{I} - \vc{R}]^{-1} \vc{e}
},
\label{eq:Rama-pi(0)}
\end{align}
where $\vc{R}_0 = \sum_{m=1}^{\infty}\vc{R}_0(m)$, $\vc{R} = \sum_{m=1}^{\infty}\vc{R}(m)$, and $\vc{\kappa}$ is the stationary distribution vector of the stochastic matrix $\vc{K}$ defined as
\begin{align*}
\vc{K}
= \vc{B}_0 + \sum_{m=1}^{\infty} \vc{B}_m \vc{G}^m.
\end{align*}

\begin{REM}
Assumption \ref{assum1} ensures that $\{(X_{\nu},J_{\nu})\}$ is irreducible and positive recurrent (see \cite[Chapter~XI, Proposition~3.1]{Asmu03}). Therefore, $\vc{K}$ is an irreducible stochastic matrix with the unique stationary distribution vector~$\vc{\kappa}$ (see \cite[Theorem~3.1]{Sche90}). In addition,  inverse $[\vc{I} - \vc{\Phi}_0]^{-1}$ exists (see the proof of \cite[Theorem~2.1~(ii)]{Sche90}) and $[\vc{I} - \vc{R}]^{-1}$ does (see \cite[Theorem~3.4]{Zhao98}).
\end{REM}

Ramaswami's recursion (\ref{eq:Rama-pi(k)}) with (\ref{eq:Rama-pi(0)}) includes the infinite sequences $\{\vc{A}_k;k \ge -1\}$ and $\{\vc{B}_k;k \ge -1\}$, and therefore the infinite sequences are {\it truncated} in implementing the recursion. More specifically, the infinite sequences $\{\vc{A}_k;k \ge -1\}$ and $\{\vc{B}_k;k \ge -1\}$ are replaced with the substantially finite sequences  $\{\vc{A}_k^{(N)};k \ge -1\}$ and $\{\vc{B}_k^{(N)}; k\ge -1\}$, where
\begin{eqnarray}
\vc{A}^{(N)}_k
&=&
\left\{
\begin{array}{ll}
\vc{A}_k,			& -1 \le k \le N-1,
\\
\ol{\vc{A}}_{N-1}:=\sum_{\ell=N}^{\infty}\vc{A}_{\ell}, 	& k=N,
\\
\vc{O}, & k\ge N+1,
\end{array}
\right.
\label{defn-A^(N)(k)}
\\
\vc{B}^{(N)}_k
&=&
\left\{
\begin{array}{ll}
\vc{B}_k,			& -1 \le k \le N-1,
\\
\ol{\vc{B}}_{N-1}:=\sum_{\ell=N}^{\infty}\vc{B}_{\ell}, 	& k=N,
\\
\vc{O}, & k\ge N+1,
\end{array}
\right.
\label{defn-B^(N)(k)}
\end{eqnarray}

The truncated sequences $\{\vc{A}_k^{(N)}\}$ and $\{\vc{B}_k^{(N)}\}$ yield a computable approximation $\vc{\pi}^{(N)} := (\vc{\pi}_0^{(N)},\vc{\pi}_1^{(N)},\vc{\pi}_2^{(N)},\dots)$ to the stationary distribution vector $\vc{\pi}=(\vc{\pi}_0,\vc{\pi}_1,\vc{\pi}_2,\dots)$. Indeed, the approximate distribution vector $\vc{\pi}^{(N)} = (\vc{\pi}_0^{(N)},\vc{\pi}_1^{(N)},\vc{\pi}_2^{(N)},\dots)$ is obtained by using $\{\vc{A}_k^{(N)}\}$ and $\{\vc{B}_k^{(N)}\}$ for $\{\vc{A}_k\}$ and $\{\vc{B}_k\}$ in Ramaswami's recursion (\ref{eq:Rama-pi(k)}) with (\ref{eqn:G-matrix}):
\begin{subequations}\label{eq:Rama-pi^{(N)}(k)}
For $k=1,2,\dots$,
\begin{alignat}{2}
\vc{\pi}_k^{(N)}
&= \vc{\pi}_0^{(N)} \vc{R}_0^{(N)}(k)
+ \sum_{\ell=1}^{k-1}\vc{\pi}_\ell^{(N)} \vc{R}^{(N)}(k-\ell),
\\
\vc{R}^{(N)}(k)
&=
\sum_{m=0}^{N-k}
\vc{A}_{k+m}^{(N)} [\vc{G}^{(N)}]^m
[\vc{I}-\vc{\Phi}_0^{(N)}]^{-1}, & ~ k &\in [1,N],
\\
\vc{R}_0^{(N)}(k)
&=
\sum_{m=0}^{N-k}
\vc{B}_{k+m}^{(N)} [\vc{G}^{(N)}]^m
[\vc{I}-\vc{\Phi}_0^{(N)}]^{-1}, & ~ k &\in [1,N],
\\
\vc{\Phi}_0^{(N)}
&= \sum_{m=0}^N \vc{A}_m^{(N)} [\vc{G}^{(N)}]^m,
\end{alignat}
\end{subequations}
where $\vc{G}^{(N)}$ is the minimal nonnegative solution of the matrix equation
\begin{align}
\vc{G}^{(N)} = \sum_{m=0}^N \vc{A}_m^{(N)}[\vc{G}^{(N)}]^{m+1};
\label{eqn:G^{(N)}-matrix}
\end{align}
and $\vc{\pi}_0^{(N)}$ is given by
\begin{align}
\vc{\pi}_0^{(N)}
= {
\vc{\kappa}^{(N)}
\over
\vc{\kappa}^{(N)}\vc{R}_0^{(N)}[\vc{I} - \vc{R}^{(N)}]^{-1}\vc{e}
},
\end{align}
where $\vc{R}_0^{(N)} = \sum_{m=1}^N\vc{R}_0^{(N)}(m)$, $\vc{R}^{(N)} = \sum_{m=1}^N\vc{R}^{(N)}(m)$, and $\vc{\kappa}^{(N)}$ is the stationary distribution vector of the stochastic matrix $\vc{K}^{(N)}$ defined as
\begin{align}
\vc{K}^{(N)}
= \vc{B}_0^{(N)} + \sum_{m=1}^N \vc{B}_m^{(N)} [\vc{G}^{(N)}]^m.
\label{defn:K^{(N)}}
\end{align}
Clearly, $\vc{\pi}^{(N)}= (\vc{\pi}_0^{(N)},\vc{\pi}_1^{(N)},\vc{\pi}_2^{(N)},\dots)$ is computable, provided that $\vc{G}^{(N)}$ is given. The matrix $\vc{G}^{(N)}$ is computed as the limit of the following sequence $\{\vc{G}^{(N)}_n : n \in \bbZ_+\}$:
\begin{equation*}
\vc{G}^{(N)}_n=
\begin{cases}
\displaystyle
\vc{O}, & \text{$n=0$,}
\\
\displaystyle
\sum_{m = -1}^N \vc{A}^{(N)}_m [\vc{G}^{(N)}_{n-1}]^{m+1},
 & \text{$n \in \bbN$.}
\end{cases}
\end{equation*}

We call $\vc{\pi}^{(N)}$ the {\it level-increment (LI) truncation approximation} to the stationary distribution vector $\vc{\pi}$. This approximation is equivalent to replacing the original stochastic matrix $\vc{P}$ with another one $\vc{P}^{(N)}$ defined as
\begin{align*}
\vc{P}^{(N)}
=
\bordermatrix{
& \bbL_0
& \bbL_1
& \bbL_2
& \bbL_3
& \cdots \cr
\bbL_0
& \vc{B}^{(N)}_0
& \vc{B}^{(N)}_1
& \vc{B}^{(N)}_2
& \vc{B}^{(N)}_3
& \cdots \cr
\bbL_1
& \vc{B}^{(N)}_{-1}
& \vc{A}^{(N)}_0
& \vc{A}^{(N)}_1
& \vc{A}^{(N)}_2
& \cdots \cr
\bbL_2
& \vc{O}
& \vc{A}^{(N)}_{-1}
& \vc{A}^{(N)}_0
& \vc{A}^{(N)}_1
& \cdots \cr
\bbL_3
& \vc{O}
& \vc{O}
& \vc{A}^{(N)}_{-1}
& \vc{A}^{(N)}_0
& \cdots
\cr
~\vdots
& \vdots
& \vdots
& \vdots
& \vdots
& \ddots
}.
\end{align*}
Equations (\ref{defn-A^(N)(k)})--(\ref{defn:K^{(N)}}) imply that the approximate distribution vector $\vc{\pi}^{(N)} = (\vc{\pi}_0^{(N)},\vc{\pi}_1^{(N)},\vc{\pi}_2^{(N)},\dots)$ is a stationary distribution vector of the M/G/1-type stochastic matrix $\vc{P}^{(N)}$. In addition, an M/G/1-type Markov chain driven by $\vc{P}^{(N)}$ has no level increments beyond $N$. Hence, we refer to $\vc{P}^{(N)}$ and its stationary distribution vector $\vc{\pi}^{(N)}$ as the {\it level-increment (LI) truncation approximations} to $\vc{P}$ and $\vc{\pi}$, respectively.

We here emphasize the following: it is worthwhile to study the error evaluation of the LI truncation approximation, more specifically, to estimate the truncation parameter $N$ of level increments such that given error tolerance is satisfied (even if the name ``LI truncation approximation" is used only in the literature \cite{Ouchi-Masuyama21}) because the approximation is a standard method for computing the stationary distribution of the M/G/1-type Markov chain.

As far as we know, there are no studies on the LI truncation approximation except for \cite{Ouchi-Masuyama21}. Under Assumption~\ref{assum1}, the vector $\vc{\pi}^{(N)}$ is the unique stationary distribution of $\vc{P}^{(N)}$ \cite[Proposition~3.1]{Ouchi-Masuyama21} and satisfies the following \cite[Theorem 4.2]{Ouchi-Masuyama21}:
\begin{equation}
\lim_{N \to \infty} \|\vc{\pi}^{(N)} - \vc{\pi}\| = 0.
\label{prop:convergence}
\end{equation}
In addition, suppose that there exist $M_1$- and $M_0$-dimensional nonnegative column vectors $\vc{c}_A$ and $\vc{c}_B$ (either of them is a non-zero vector) and a long-tailed distribution $F$ on $\bbZ_+$ (see, e.g., \cite{Foss13} for the definition of long-tailed distributions) such that
\begin{align*}
\lim_{N \to \infty}
{
\ool{\vc{A}}_N\vc{e}
\over
\ol{F}(N)
} = \vc{c}_A,
\quad
\lim_{N \to \infty}
{
\ool{\vc{B}}_N\vc{e}
\over
\ol{F}(N)
} = \vc{c}_B,
\end{align*}
where
\begin{align*}
\ool{\vc{A}}_k
&= \sum_{\ell = k + 1}^\infty \ol{\vc{A}}_\ell,\quad
\ool{\vc{B}}_k
= \sum_{\ell = k + 1}^\infty \ol{\vc{B}}_\ell,
\quad k \in \bbZ_+,
\end{align*}
We then have the subgeometric convergence formula (\cite[Theorem 5.2]{Ouchi-Masuyama21}): For $k \in \bbZ_+$,
\begin{align*}
\lim_{N \to \infty}
{
\vc{\pi}_0^{(N)} - \vc{\pi}_k
\over
\ol{F}(N)
}
&=
{\vc{\pi}_0\vc{c}_B + \ol{\vc{\pi}}_0\vc{c}_A
\over
-\sigma
}\vc{\pi}_k>\vc{0},
\end{align*}
where $\ol{\vc{\pi}}_0 = \sum_{\ell = 1}^\infty \vc{\pi}_{\ell}$.

This paper considers the complementary case to the one studied in \cite{Ouchi-Masuyama21}. The main contribution of this paper is to present a geometric convergence formula for the level-wise difference $\vc{\pi}^{(N)}_k - \vc{\pi}_k$, assuming that the level-increment distribution is light-tailed. The geometric convergence formula shows the decay rate of $\vc{\pi}^{(N)}_k - \vc{\pi}_k$ is equivalent to the tail decay rate of the integrated-tail distribution of level increments. This decay rate equivalence is consistent with what is inspired by the subgeometric convergence formula in \cite{Ouchi-Masuyama21}.

The rest of this paper consists of two sections. Section~\ref{sec:main_results} presents the main theorem on the geometric convergence formula for the level-wise difference $\vc{\pi}^{(N)}_k - \vc{\pi}_k$, together with its corollaries. Section~\ref{sec:concluding} contains concluding remarks.

\section{Main results}\label{sec:main_results}

This section consists of two subsections. Section~\ref{sec:definitions}
provides basic definitions and some assumptions for the main theorem of this paper. Section~\ref{subsec:main_thm} presents the main theorem, together with its corollaries, on a geometric convergence formula for the level-wise difference between the stationary distribution vector and its LI truncation approximation.

\subsection{Basic definitions and assumptions}
\label{sec:definitions}

We provide some definitions concerned with vectors and matrices. Let $\|\vc{x}\| = \sum_{i} |x_i|$ for any vector $\vc{x}:=(x_i)$. For any matrix function $\vc{Z}(\,\cdot\,):=(Z_{i,j}(\,\cdot\,))$ and scalar function $f(\,\cdot\,)$ on $(-\infty,\infty)$, write $\vc{Z}(x) = \ol{\bcal{O}}(f(x))$ to represent
\begin{align*}
\limsup_{x\to\infty} { \sup_i \sum_{j}|Z_{i,j}(x)| \over f(x)}
< \infty.
\end{align*}
In addition, for any matrix $\vc{X} \ge \vc {O}$, write $\vc{X} < \infty$ when each element of $\vc{X}$ is finite.

Next, we introduce three assumptions for our geometric convergence formula for the level-wise difference $\vc{\pi}^{(N)}_k-\vc{\pi}_k$.
\begin{ASS}\label{assum:aperiodic}
The single communication class of $\vc{G}$ is aperiodic (primitive), and thus (see, e.g., \cite[Theorem~8.5.1]{Horn13}) there exists some $\varepsilon > 0$ such that
\begin{equation}
\vc{G}^n = \vc{e}\vc{g} + \ol{\bcal{O}}((1+\varepsilon)^{-n}).
\label{eq:G-ergodic}
\end{equation}
\end{ASS}

\begin{REM}\label{rem:G_aperiodic}
Assumption~\ref{assum:aperiodic} is equivalent to the aperiodicity of a Markov additive process (MAdP) $\{(\bv{X}_n, \bv{J}_n); n\in\bbZ_+\}$ with kernel $\{\vc{A}_k; k\in\bbZ\}$, where $\vc{A}_k= \vc{O}$ for $k=-2, -3, \ldots$ (see \cite[Proposition 2.3]{Masu11} and \cite[Proposition~2.5.2]{Kimu13}).
\end{REM}

\begin{ASS}\label{assum:light-tailed}
Let
\begin{align*}
r_A &= \sup\left\{z \ge 1: \sum_{k=1}^{\infty} z^k \vc{A}_k < \infty\right\},
\\
r_B &= \sup\left\{z \ge 1: \sum_{k=1}^{\infty} z^k \vc{B}_k < \infty\right\},
\end{align*}
and assume that $r := \min (r_A, r_B) > 1$, and thus the increment of the level process $\{X_n\}$ is light-tailed.
\end{ASS}

\begin{ASS}\label{assum3}
There exists a nonnegative function $f$ on $\bbZ_+$ such that
\begin{align}
\lim_{N \to \infty}
{
\ool{\vc{A}}_N\vc{e}
\over
r^{-N}f(N)
} &= \vc{c}_A,
\quad
\lim_{N \to \infty}
{
\ool{\vc{B}}_N\vc{e}
\over
r^{-N}f(N)
} = \vc{c}_B,
\label{eq:assum_geometric}
\\
\lim_{N \to \infty}
{
f(N)
\over
f(N-1)
} &= 1,
\label{eq:assum_geometric_2}
\end{align}
where $\vc{c}_A \ge \vc{0}$ and $\vc{c}_B \ge \vc{0}$ are $M_1$- and $M_0$-dimensional finite column vectors, respectively, and either of them is a non-zero vector.
\end{ASS}

\begin{REM}\label{rem:c_A_and_c_B}
Assumption~\ref{assum:light-tailed} implies that (i) if $r_A>r_B$ then $\vc{c}_A=\vc{0}$; and (ii) if $r_A < r_B$ then $\vc{c}_B=\vc{0}$.
\end{REM}
%

%
\begin{REM}\label{rem:example_of_f(N)}
The function $f$ of Assumption~\ref{assum3} is introduced to expand the applicability of the main theorem (Theorem~\ref{th:geo}). For example, if the function $f$ was removed from Assumption~\ref{assum3}, Assumption~\ref{assum3} could not cover M/G/1-type Markov chains with discrete gamma-distributed level increments; that is, with $\{\vc{A}_k\}$ and $\{\vc{B}_k\}$ such that
\begin{alignat*}{2}
\vc{A}_k &= \vc{C}_A k^{\alpha - 1} \gamma_A^k, &\quad k & \in \bbZ_+,
\\
\vc{B}_k &= \vc{C}_B k^{\beta - 1} \gamma_B^k, &\quad k & \in \bbZ_+,
\end{alignat*}
where $\alpha,\beta > 0$, $\gamma \in (0,1)$, and $\vc{C}_A$ and $\vc{C}_B$ are some nonnegative matrices. Naturally, $f$ is allowed to be of other types than power type. Shown below are typical examples of the function $f$ satisfying \eqref{eq:assum_geometric_2}:
\begin{enumerate}
\item $f$ is of power type, i.e., $f(k) = k^{\alpha}$ with $\alpha \in (-1, \infty)$.
\item $f$ is logarithmic, i.e., $f(k) = \log (k + 1)$.
\item $f$ is constant, i.e., $f(k) = c$ with $c \in (0, \infty)$.
\end{enumerate}
\end{REM}
\begin{REM}\label{rem:r^Nf(N)_to_0}
Equation \eqref{eq:assum_geometric_2} implies that
\begin{align*}
\lim_{N \to \infty} {r^{-N} f(N) \over r^{-(N-1)} f(N-1)} = r^{-1},
\end{align*}
that is, $r^{-N} f(N)$ eventually converges to zero at a rate as $r^{-N}$.
\end{REM}

\subsection{The main theorem and its corollaries}\label{subsec:main_thm}

In this subsection, we first present the main theorem, which is a geometric convergence formula for the level-wide difference $\vc{\pi}^{(N)}_k-\vc{\pi}_k$ between the stationary distribution vector $\vc{\pi}$ and its LI truncation approximation $\vc{\pi}^{(N)}$. We then provide two corollaries of the theorem. The first one is concerned with the decay rate of the relative error of $\vc{\pi}^{(N)}_k$ to $\vc{\pi}_k$. The second one relates the decay rates of the level-wide difference and relative error to the decay rate of the level-increment distribution.

The following is the main theorem of this paper.
\begin{THM}\label{th:geo}
If Assumptions \ref{assum1} and \ref{assum:aperiodic}--\ref{assum3} hold,
then
\begin{alignat}{2}
\lim_{N \to \infty}
{
\vc{\pi}^{(N)}_k - \vc{\pi}_k
\over
r^{-N}f(N)
}
&=
{r(\vc{\pi}_0\vc{c}_B + \ol{\vc{\pi}}_0\vc{c}_A)
\over
-\sigma
}\vc{\pi}_k>\vc{0}, &\quad k &\in \bbZ_+.
\label{eq:th-geo}
\end{alignat}
\end{THM}
\proof
See \ref{sec:PFOF_th-geo}. \qed

\medskip

The first corollary below shows that the relative error of $\vc{\pi}^{(N)}_k$ to $\vc{\pi}_k$ is asymptotically independent of the level variable $k$.
\begin{COR}\label{cor:geo-01}
If all the conditions of Theorem~\ref{th:geo} hold, then, for $k \in \bbZ_+$,
\begin{align}
\lim_{N \to \infty}
{1 \over r^{-N}f(N)}
\left\|
{
\vc{\pi}^{(N)}_k - \vc{\pi}_k
\over
\vc{\pi}_k\vc{e}
}
\right\|
&=
{r(\vc{\pi}_0\vc{c}_B + \ol{\vc{\pi}}_0\vc{c}_A)
\over
-\sigma
}>0.
\label{eq:geo-norm-01}
\end{align}
\end{COR}

\begin{proof}
Equation \eqref{eq:th-geo} implies that for each $k \in \bbZ_+$ there exists $N_k \in \bbN$ such that
\[
\vc{\pi}^{(N)}_k - \vc{\pi}_k > \vc{0},
\qquad N \ge N_k,
\]
and thus
\[
\left\|
{
\vc{\pi}^{(N)}_k - \vc{\pi}_k
\over
\vc{\pi}_k\vc{e}
}
\right\|
=
{
(\vc{\pi}^{(N)}_k - \vc{\pi}_k)\vc{e}
\over
\vc{\pi}_k\vc{e}
}, \qquad N \ge N_k.
\]
Combining this and \eqref{eq:th-geo} leads to \eqref{eq:geo-norm-01}.  \qed
\end{proof}

\medskip

To describe the second corollary, we introduce a distribution associated with level increments. Let
\begin{align*}
D(k) = \sum_{(\ell,i) \in \bbS} \pi_{\ell,i}
\PP( \Delta_+ \le k \mid (X_0,J_0)=(\ell,i)),
\quad
k \in \bbZ_+,
\end{align*}
where $\Delta_+ = \max(X_1-X_0,0)$. We then have
\begin{align}
D(k)
&= \sum_{(\ell,i) \in \bbS} \pi_{\ell,i}
\PP( X_1-X_0 \le k \mid (X_0,J_0)=(\ell,i))
\nonumber
\\
&= \sum_{n=0}^k
\vc{\pi}_0 \vc{B}_n\vc{e}
+ \sum_{n=-1}^k \ol{\vc{\pi}}_0\vc{A}_n\vc{e},
\quad k \in \bbZ_+.
\label{defn:D}
\end{align}
The distribution $D$ is referred to as the {\it stationary nonnegative level-increment (SNL) distribution} \cite[Section~5.3.1]{Masu21-M/G/1-Subexp}. Let $D_I$ denote the integrated-tail distribution (equilibrium distribution) of the SNL distribution, that is, for $k  \in \bbZ_+$,
\begin{align}
D_I(k)
&= {\sum_{\ell=0}^k (1 - D(\ell)) \over \sum_{\ell=1}^{\infty}\ell D(\ell)},
\quad k \in \bbZ_+.
\label{defn:D_I(k)}
\end{align}

The second corollary is obtained by combining Theorem~\ref{th:geo} with Corollary~\ref{cor:geo-01}.
\begin{COR}\label{coro:geo_ELID}
Suppose that all the conditions of Theorem~\ref{th:geo} hold, and let $\ol{D}_I(k) = 1 - D_I(k)$ for $k \in \bbZ_+$. We then have, for $k \in \bbZ_+$,
\begin{alignat}{2}
\lim_{N\to\infty}
{
\vc{\pi}^{(N)}_k - \vc{\pi}_k
\over
\ol{D}_I(N)
}
&=
{r(\vc{\pi}_0 \ol{\vc{m}}_{B} + \ol{\vc{\pi}}_0 \ol{\vc{m}}_{A}^+)
\over -\sigma}
\vc{\pi}_k>\vc{0},
\label{eq:th-geo-2}
\\
\lim_{N\to\infty}
{1 \over
\ol{D}_I(N)
}
{
\left\|
\vc{\pi}^{(N)}_k - \vc{\pi}_k
\over
\vc{\pi}_k\vc{e}
\right\|
}
&=
{r(\vc{\pi}_0 \ol{\vc{m}}_{B} + \ol{\vc{\pi}}_0 \ol{\vc{m}}_{A}^+)
 \over -\sigma}>0,
\label{eq:geo-norm-02}
\end{alignat}
where $\ol{\vc{m}}_A^+ = \sum_{k=1}^\infty k\vc{A}_k\vc{e} = \ol{\vc{m}}_A + \vc{A}_{-1}\vc{e}$.
\end{COR}
\proof
It follows from \eqref{defn:D} and \eqref{defn:D_I(k)} that
\begin{align*}
\ol{D}_I(k)
&= \sum_{\ell=k+1}^{\infty}
{
\vc{\pi}_0 \ol{\vc{B}}_\ell\vc{e} + \ol{\vc{\pi}}_0 \ol{\vc{A}}_\ell\vc{e}
\over
\vc{\pi}_0 \ol{\vc{m}}_{B} + \ol{\vc{\pi}}_0 \ol{\vc{m}}_{A}^+
}
=
{
\vc{\pi}_0 \ool{\vc{B}}_k\vc{e} + \ol{\vc{\pi}}_0 \ool{\vc{A}}_k\vc{e}
\over
\vc{\pi}_0 \ol{\vc{m}}_{B} + \ol{\vc{\pi}}_0 \ol{\vc{m}}_{A}^+
},\quad k \in \bbZ_+.
\end{align*}
Applying Assumption~\ref{assum3} to this equation yields
\begin{align}
\lim_{k\to\infty}{\ol{D}_I(k) \over r^{-k}f(k)}
= {
\vc{\pi}_0 \vc{c}_{B} + \ol{\vc{\pi}}_0 \vc{c}_{A}
\over
\vc{\pi}_0 \ol{\vc{m}}_{B} + \ol{\vc{\pi}}_0 \ol{\vc{m}}_{A}^+
} \in (0,\infty).
\label{asymp:F_pi^I}
\end{align}
Combining \eqref{eq:th-geo} with \eqref{asymp:F_pi^I} leads to
\begin{align*}
\lim_{N\to\infty}
{
\vc{\pi}^{(N)}_k - \vc{\pi}_k
\over
\ol{D}_I(N)
} &=
\lim_{N\to\infty}
{
\vc{\pi}^{(N)}_k - \vc{\pi}_k
\over
r^{-N}f(N)
}\cdot
{
r^{-N}f(N)
\over
\ol{D}_I(N)
}
\\
&=
{
r(\vc{\pi}_0\vc{c}_B + \ol{\vc{\pi}}_0\vc{c}_A)
\over
-\sigma
}\vc{\pi}_k \cdot
{
\vc{\pi}_0 \ol{\vc{m}}_{B} + \ol{\vc{\pi}}_0 \ol{\vc{m}}_{A}^+
\over
\vc{\pi}_0 \vc{c}_{B} + \ol{\vc{\pi}}_0 \vc{c}_{A}
}
\\
&=
{r(\vc{\pi}_0 \ol{\vc{m}}_{B} + \ol{\vc{\pi}}_0 \ol{\vc{m}}_{A}^+)
\over
-\sigma}
\vc{\pi}_k,
\end{align*}
which shows that (\ref{eq:th-geo-2}) holds. Similarly, (\ref{eq:geo-norm-02}) follows from (\ref{eq:geo-norm-01}) and (\ref{asymp:F_pi^I}).
 \qed

\medskip

Corollary~\ref{coro:geo_ELID} shows that, the decay rate of $\vc{\pi}^{(N)}_k-\vc{\pi}_k$ is asymptotically equal to the integrated tail distribution $D_I$ of the SNL distribution $D$, provided that the increment of the level process $\{X_n\}$ is light-tailed (see Assumption~\ref{assum:light-tailed}). A similar decay rate equivalence is reported in the subgeometric convergence case (see \cite[Corollary 5.3]{Ouchi-Masuyama21}).

\section{Concluding remarks}\label{sec:concluding}
The main theorem (Theorem~\ref{th:geo}) of this paper presents the geometric convergence formula \eqref{eq:th-geo} for the LI truncation approximation to the stationary distribution vector of the M/G/1-type Markov chain.

Based on the main theorem and its corollaries, we can estimate a value of the truncation parameter $N$ satisfying the given error tolerance. Indeed, it follows from (\ref{eq:geo-norm-01}) that, for all $k \in \bbZ_+$ and sufficiently large $N \in \bbN$,
\begin{align}
\left\|
{
\vc{\pi}^{(N)}_k - \vc{\pi}_k
\over
\vc{\pi}_k\vc{e}
}
\right\|
&\approx
{(\vc{\pi}_0\vc{c}_B + \ol{\vc{\pi}}_0\vc{c}_A)
\over
-\sigma
}
r^{-N+1}f(N)
\nonumber
\\
&\le { c^* \over -\sigma }r^{-N+1}f(N),
\label{eqn:2022_0802-01}
\end{align}
where $c^* > 0$ is the maximum among all the elements of $\vc{c}_A$ and $\vc{c}_B$. Thus, we can use a value
\begin{align*}
N^*
:= \min\left\{N \in \bbN: { c^* \over -\sigma }r^{-N+1}f(N) < \varep\right\}
\end{align*}
as the truncation parameter $N$ of the LI truncation approximation, given that an error tolerance $\varep > 0$ for the relative error of $\vc{\pi}^{(N)}_k$ to $\vc{\pi}_k$.

However, there are two problems in the above argument on estimating the truncation parameter $N$: (i) The error control of the approximate distribution is for each level but not for the distribution as a whole; (ii) the first equality in error evaluation equation (\ref{eqn:2022_0802-01}) is not an exact but approximate one.

We remark on the first problem. Suppose that the dominated convergence theorem is available for our present problem. It then follows from the {\it level-wise} convergence formula (\ref{eq:geo-norm-01}) that
\begin{align}
&
\lim_{N \to \infty}
{1 \over r^{-N}f(N)}
\left\|
\vc{\pi}^{(N)} - \vc{\pi}
\right\|
\nonumber
\\
&~=
\lim_{N \to \infty}
{1 \over r^{-N}f(N)}
\sum_{k=0}^{\infty}
\left\|
{
\vc{\pi}^{(N)}_k - \vc{\pi}_k
\over
\vc{\pi}_k\vc{e}
}
\right\|
\vc{\pi}_k\vc{e}
\nonumber
\\
&~=
\sum_{k=0}^{\infty}
\lim_{N \to \infty}
{1 \over r^{-N}f(N)}
\left\|
{
\vc{\pi}^{(N)}_k - \vc{\pi}_k
\over
\vc{\pi}_k\vc{e}
}
\right\|
\vc{\pi}_k\vc{e}
\nonumber
\\
&~=
\sum_{k=0}^{\infty}
{r(\vc{\pi}_0\vc{c}_B + \ol{\vc{\pi}}_0\vc{c}_A)
\over
-\sigma
}\vc{\pi}_k\vc{e}
\nonumber
\\
&~=
{r(\vc{\pi}_0\vc{c}_B + \ol{\vc{\pi}}_0\vc{c}_A)
\over
-\sigma
}.
\label{eqn:2022_0802-02}
\end{align}
Thus, we can obtain a {\it total-variation} convergence formula. Unfortunately, we have, at present, no idea how to justify the order exchange between the limit and the infinite sum in the second equality of (\ref{eqn:2022_0802-02}). This problem is one of the future tasks. However, as for the subgeometric convergence case, we identify a sufficient condition under which such a total-variation convergence formula like (\ref{eqn:2022_0802-02}) (see \cite{Ouchi-Masuyama22}).

We close this section with a remark on the second problem. Exact and conservative error evaluation (in both geometric and subgeometric convergence cases) requires a computable upper bound for the total variation distance between the original stationary distribution and its LI truncation approximation. Such an upper bound may be derived by the Foster Lyapunov drift condition, as is done in \cite{LLM2018,Masu15-AAP,Masu16-SIMAX,Masu17-LAA,Masu17-JORSJ} for another approximation to countable-state Markov chains.

\appendix


\section{Proof of Theorem~\ref{th:geo}}
\label{sec:PFOF_th-geo}
To prove Theorem~\ref{th:geo}, we use the following difference formula between $\vc{\pi}^{(N)}_k$ and $\vc{\pi}_k$.
\begin{PROP}[{}{\cite[Lemma 4.1]{Ouchi-Masuyama21}}]\label{prop:difference_formula}
If Assumption \ref{assum1} holds, then
\begin{align}
&\vc{\pi}^{(N)}_k - \vc{\pi}_k
\nonumber
\\
&~= \vc{\pi}^{(N)}_0\Biggl[
\frac{1}{-\sigma}
\ool{\vc{B}}_{N - 1}\vc{e}\vc{\pi}_k
+ \sum_{n=N+1}^{\infty} \vc{B}_n
\left(\vc{G}^{N-k} - \vc{G}^{n-k}\right)\vc{F}_+(k; k)
\Biggr.
\nonumber
\\
&\quad~
+ \sum_{n=N+1}^{\infty} \vc{B}_n (\vc{G}^{N-1}- \vc{G}^{n-1})\vc{S}(k)
\Biggr]
\nonumber
\\
&~~
+ \!\sum_{\ell=1}^{\infty} \vc{\pi}^{(N)}_\ell
\Biggl[
\frac{1}{-\sigma}\ool{\vc{A}}_{N - 1}\vc{e}\vc{\pi}_k
+ \!\!\!\sum_{n=N+1}^{\infty}\!\! \vc{A}_n
\left( \vc{G}^{N + \ell - k} - \vc{G}^{n + \ell - k}\right)\vc{F}_+(k;k)
\Biggr.
\nonumber
\\
&{} \qquad
+ \Biggl. \sum_{n=N+1}^{\infty} \vc{A}_n
\left(\vc{G}^{N + \ell - 1} - \vc{G}^{n + \ell - 1}\right)\vc{S}(k)\Biggr],
\quad 0 \le k \le N,
\label{eq:lem}
\end{align}
where
\begin{alignat*}{2}
\vc{S}(k) &= [\vc{I}-\vc{\Phi}_0]^{-1}\vc{B}_{-1}\vc{H}(0; k)
\\
&\quad~+ \vc{G}(\vc{I} -\vc{A} -\ol{\vc{m}}_A\vc{g})^{-1}\vc{e}\vc{\pi}_k,
\qquad
k \in \mathbb{Z}_+.
\end{alignat*}
Note here that $\vc{H}(k;\ell) := (H(k,i ; \ell,j))_{(i,j)\in(\bbM_{1\vmin k}\times\bbM_{1\vmin \ell})}$ for $k, \ell \in \bbZ_+$ and $\vc{F}_+(k;\ell) = (F_+(k,i; \ell,j))_{(i,j) \in (\bbM_1)^2}$ for $k, \ell \in \bbN$ are respectively defined by
\begin{align*}
H(k,i;\ell,j) &= \EE_{(k,i)}\!
\left[
\sum_{\nu=0}^{T_{\{(k_*,i_*)\}}-1} \dd{1}((X_{\nu},J_{\nu})=(\ell,j))
\right]
\\
&\quad~- \pi_{\ell,j}\EE_{(k,i)}
\left[ T_{\{(k_*,i_*)\}} \right],
\\
F_+(k,i;\ell,j)
&=\EE_{(k,i)}\!
\left[
\sum_{\nu=0}^{T_{\bbL_0}-1} \dd{1}((X_{\nu},J_{\nu})=(\ell,j))
\right],
\end{align*}
where $(k_*, i_*)\in \bbS$ is any fixed state and where $T_{\bbA} = \inf\{\nu \in \bbN: (X_{\nu},J_{\nu}) \in \bbA\}$ for any subset $\bbA \subset \bbS$.
\end{PROP}

\PFOF{Theorem~\ref{th:geo}}
For simplicity, we rewrite the difference formula \eqref{eq:lem}. To this end, let $\vc{S}_0(k) = \vc{F}_+(k; k) + \vc{G}^{k-1}\vc{S}(k)$. It then follows from \eqref{eq:lem} that, for $0 \le k \le N$,
\begin{align}
&\vc{\pi}^{(N)}_k - \vc{\pi}_k
\nonumber
\\
&~=
\frac{1}{-\sigma}
\Biggl[
\vc{\pi}^{(N)}_0
\ool{\vc{B}}_{N - 1}\vc{e}
+ \sum_{\ell=1}^{\infty} \vc{\pi}^{(N)}_\ell
\ool{\vc{A}}_{N - 1}\vc{e}
\Biggr]
\vc{\pi}_k
\nonumber
\\
&\quad {} +
\vc{\pi}^{(N)}_0
\sum_{n=N+1}^{\infty} \vc{B}_n
(\vc{G}^{N-k}- \vc{G}^{n-k})\vc{S}_0(k)
\nonumber
\\
&\quad {} +
\sum_{\ell=1}^{\infty} \vc{\pi}^{(N)}_\ell
\sum_{n=N+1}^{\infty}\vc{A}_n
\left(\vc{G}^{N + \ell - k} - \vc{G}^{n + \ell - k}\right)\vc{S}_0(k).
\label{eq:piN-pi_S0}
\end{align}

First, we consider the asymptotics of the first term of (\ref{eq:piN-pi_S0}).
From Assumption~\ref{assum3}, we have
\begin{align*}
\lim_{N\to\infty} {\ool{\vc{B}}_{N-1}\vc{e} \over r^{-N}f(N)}
&= \lim_{N\to\infty} r {\ool{\vc{B}}_{N-1}\vc{e} \over r^{-(N-1)}f(N-1)} {f(N-1) \over f(N)}
= r\vc{c}_B,
\\
\lim_{N\to\infty} {\ool{\vc{A}}_{N-1}\vc{e} \over r^{-N}f(N)}
&= \lim_{N\to\infty} r {\ool{\vc{A}}_{N-1}\vc{e} \over r^{-(N-1)}f(N-1)} {f(N-1) \over f(N)}
= r\vc{c}_A.
\end{align*}
Using these equations and (\ref{prop:convergence}), we obtain
\begin{align}
&
\lim_{N\to\infty}
{1 \over r^{-N}f(N)}
\frac{1}{-\sigma}
\Biggl[
\vc{\pi}^{(N)}_0
\ool{\vc{B}}_{N - 1}\vc{e}
+ \sum_{\ell=1}^{\infty} \vc{\pi}^{(N)}_\ell
\ool{\vc{A}}_{N - 1}\vc{e}
\Biggr]
\vc{\pi}_k
\nonumber
\\
& {}
\quad= {r \over -\sigma}
\left(
\vc{\pi}_0 \vc{c}_B
+ \ol{\vc{\pi}}_0 \vc{c}_A
\right)\vc{\pi}_k,\quad k \in \bbZ_+.
\label{eq:decay_rate_ool}
\end{align}

Next, we evaluate the second term of (\ref{eq:piN-pi_S0}). Equation \eqref{eq:G-ergodic} implies that
\begin{align}
&\left|\vc{G}^{N-k} - \vc{G}^{n-k}\right|
\nonumber
\\
&\quad\le \ol{\bcal{O}}\left((1+\varepsilon)^{-N+k}\right) + \ol{\bcal{O}}\left((1+\varepsilon)^{-n+k}\right)
\nonumber
\\
&\quad= \ol{\bcal{O}}\left((1+\varepsilon)^{-N+k}\right), \quad 0 \le k \le N,~~ n \ge N+1.
\label{eq:Gn-GN}
\end{align}
Furthermore, Assumption~\ref{assum3} ensures that
\begin{equation}
\lim_{N\to\infty} {\sum_{n=N+1}^\infty \vc{B}_n\vc{e} \over r^{-N}f(N)}
= \lim_{N\to\infty} {\ool{\vc{B}}_{N-1}\vc{e} - \ool{\vc{B}}_N\vc{e} \over r^{-N}f(N)}
= (r-1)\vc{c}_B.
\label{eq:sumB/r^-N}
\end{equation}
Combining \eqref{eq:Gn-GN} and \eqref{eq:sumB/r^-N} yields
\begin{equation}
\lim_{N\to\infty}
{\sum_{n=N+1}^\infty \vc{B}_n(\vc{G}^{N-k} - \vc{G}^{n-k})\vc{S}_0(k)
\over
r^{-N}f(N)}
= \vc{O}, \quad k \in \bbZ_+.
\label{eq:ol_BG/r^(-N)_to_0}
\end{equation}

We then consider the third term of (\ref{eq:piN-pi_S0}). As in \eqref{eq:sumB/r^-N}, we have
\begin{equation}
\lim_{N\to\infty} {\sum_{n=N+1}^\infty \vc{A}_N\vc{e} \over r^{-N}f(N)}
= (r-1)\vc{c}_A.
\label{eq:sumA/r^-N}
\end{equation}
From \eqref{eq:Gn-GN}, we also obtain
\begin{align}
&\sum_{\ell=1}^\infty |\vc{G}^{-N-\ell+k} - \vc{G}^{-n-\ell+k}|
\nonumber
\\
&\quad\le \sum_{\ell=1}^\infty \ol{\bcal{O}}\left((1+\varepsilon)^{-N-\ell+k}\right)
\nonumber
\\
&\quad= \ol{\bcal{O}}\left((1+\varepsilon)^{-N+k}\right),  \quad 0 \le k \le N,~~ n \ge N+1,
\label{eq:GNell-Gnell}
\end{align}
where the last equality is due to
\[
\sum_{\ell=1}^\infty (1+\varepsilon)^{-N-\ell+k} = \varepsilon^{-1}(1+\varepsilon)^{-N+k},
 \quad 0 \le k \le N.
\]
Note here that $\vc{\pi}^{(N)}_\ell \le \vc{e}$ for $\ell \in \bbN$. Therefore, applying the dominated convergence theorem, \eqref{eq:sumA/r^-N}, and \eqref{eq:GNell-Gnell} to the third term of (\ref{eq:piN-pi_S0}) yields, for $k \in \bbZ_+$,
\begin{equation}
\lim_{N\to\infty}
\sum_{\ell=1}^\infty \vc{\pi}^{(N)}_\ell
{\sum_{n=N+1}^\infty \vc{A}_n(\vc{G}^{N+\ell-k} - \vc{G}^{n+\ell-k})\vc{S}_0(k)
\over
r^{-N}f(N)}
=\vc{O}.
\label{eq:ol_AG/r^(-N)_to_0}
\end{equation}

Finally, we obtain \eqref{eq:th-geo} by combining (\ref{eq:piN-pi_S0}), \eqref{eq:decay_rate_ool}, \eqref{eq:ol_BG/r^(-N)_to_0}, and \eqref{eq:ol_AG/r^(-N)_to_0}. The proof is completed.
\section*{Acknowledgments}
The research of Hiroyuki Masuyama was supported in part by JSPS KAKENHI Grant Number JP21K11770.



\end{document}